\DeclareSymbolFont{AMSb}{U}{msb}{m}{n} 			
\documentclass[12pt,oneside,noamsfonts,a4paper]{amsart}

 \usepackage{xcolor}
\usepackage{todonotes}

\usepackage[utf8]{inputenx}  
\usepackage[T1]{fontenc}     

\usepackage[USenglish]{babel}

\usepackage[charter,expert,sfscaled]{mathdesign}  
\renewcommand{\in}{\smallin}
\renewcommand{\notin}{\notsmallin}
\renewcommand{\setminus}{\smallsetminus}

\usepackage[scaled=.96,sups]{XCharter}            
\usepackage[kerning]{microtype}
\DeclareMicrotypeAlias{XCharter-TLF}{bch} 

\usepackage{mathtools}       
\usepackage{url}             

\usepackage{mathscinet}			

{\newlength\figurewidth}

\newtheorem{theorem}{Theorem}
\newtheorem{proposition}[theorem]{Proposition}
\newtheorem{lemma}[theorem]{Lemma}

\newtheorem*{claim*}{Claim}
\newcommand{\claimdone}{\hfill$\blacksquare$\par}

\newtheorem{corollary}[theorem]{Corollary}

\theoremstyle{definition}
\newtheorem{definition}[theorem]{Definition}

\newcommand{\os}{\mleft\{ \,}             
\newcommand{\cs}{\, \mright\}}             
\newcommand{\set}[1]{\os#1\cs}                
\usepackage{mleftright}                       
\newcommand{\card}[1]{\mleft| #1 \mright|}    

\renewcommand{\restriction}{\mathbin{\!\upharpoonright}}   

\renewcommand{\mid}{\shortmid}             

\title{Colors of the Pseudotree}

\author{David Chodounský}
\address{Institute of Mathematics of the Czech Academy of Sciences,
Žitná~25, Praha~1, Czech Republic, and
Department of Applied Mathematics (KAM), Charles University, Malostranské náměstí~25, Praha~1,
Czech Republic.
}
\email{chodounsky@math.cas.cz}

\author{Monroe Eskew}
\address{
Kurt G\"odel Research Center,
University of Vienna,
Kolingasse 14-16,
1090 Vienna, Austria.
}
\email{monroe.eskew@univie.ac.at}

\author{Thilo Weinert}
\address{
Kurt G\"odel Research Center and Dipartimento di Scienze Matematiche, Informatiche e Fisiche (DMIF),
Universit\`a degli Studi di Udine,
via delle Scienze 206,
33100 Udine,
Friuli,
Italy.
}
\email{aleph\_omega@posteo.net}


\begin{document}

\begin{abstract}
    \noindent
    We investigate big Ramsey degrees of finite substructures of the universal countable homogeneous meet-tree and its binary variant. We prove that structures containing antichains have infinite big Ramsey degrees, and the big Ramsey degree of a 2-element chain is at least 8 and 7 for the binary variant.
    We deduce that the generic $C$-relation does not have finite big Ramsey degrees. 
\end{abstract}

\maketitle

\section{Big Ramsey degrees}

\noindent
We will provide examples of Ramsey classes for which the 
Fraïssé limit does not have finite big Ramsey degrees: the class of finite meet-trees 
and the class of finite $C$-relations. 
This contrasts with the Ramsey property of these classes
(proved in~\cite{MR437370,bodirsky_diana,Bodirsky+}).
So far only a few examples of this phenomenon have been discovered; 
the results in~\cite{Norbert,omega_labeled} is basically the complete list.
Our main technique, which uses counting oscillations of functions,
also found application in the upcoming paper~\cite{osc_Boolean} 
on big Ramsey degrees of Boolean algebras.

The notion of a big Ramsey degree was first explicitly isolated in~\cite{KPT}, and
this paper started a surge of results in this area;
see~\cite{JanAndy_survey,Natasha_survey} for recent surveys. 
Suppose $A,B$ are model-theoretic structures. 
Denote the set of all substructures of $B$ that are isomorphic to $A$ as $\binom{B}{A}$.
For $\ell \in \omega$ write $B \to {(B)}^A_{\ell}$ if for every coloring function 
$c \colon \binom{B}{A} \to \ell $ there exists $C \in \binom{B}{B}$ such that 
$c\left[\binom{C}{A}\right] \neq \ell$.
We call the smallest $\ell$ for which $B \to {(B)}^A_{\ell+1}$ the 
\emph{big Ramsey degree} of $A$ in $B$, write $\mathrm d(B : A) = \ell$. 
If no such $\ell \in \omega$ exists, let $\mathrm d(B : A) = \infty$.

We use the standard set-theoretic notation. 
We identify a natural number $\ell$ with the set $\set{0 , \dots, \ell -1}$. 
Given set $x$ and function $g$, we denote the range of $g \restriction x$ 
as $g[x]$.

\section{The structure of universal homogeneous pseudotrees}


\begin{definition}
    A \emph{pseudotree} is structure $T$ with an order (reflexive, antisymmetric, transitive) relation ${\leq}$ and a binary function ${\wedge}$ satisfying $\sup\set{x \in T \mid x \leq a, x \leq b} = a \wedge b$ for each $a, b \in T$  
    (in particular we demand this supremum to exist). 
    Moreover, for each $a \in T$ the set $D_T(a) = \set{x \in T \mid x \leq a}$ is linearly ordered by ${\leq}$.
\end{definition}

An infinite pseudotree might not to be a tree, 
since the set $D_T(a)$ does not need to be well-ordered. 
As usual, we denote the derived strict version of the ${\leq}$ order by ${<}$.
We say that a pseudotree $T$ is \emph{binary} if for every $a \in T$ and $x, y, z > a$ it is not the case that $a = x \wedge y = x \wedge z = y \wedge z$.
In particular, if $T$ is finite, then $T$ is a tree equipped with the meet operation, and $T$ is a binary pseudotree if it is a binary tree.

It is straightforward to check that the class of finite pseudotrees as well as the class of finite binary pseudotrees (with structure respecting embeddings) form a Fraïssé class.
We denote the Fraïssé limit of finite pseudotrees by $\psi_\omega$ and the Fraïssé limit of finite binary pseudotrees by $\psi_2$.
We write just $\psi$ in cases where the given argument works indifferently for both of these objects, 
and we use the same letter to also denote the domain of the structure~$\psi$.

These objects are also called universal homogeneous meet-trees in the literature; see, e.g.~\cite{tomek}.

\begin{proposition}
The pseudotrees $\psi_\omega$ and $\psi_2$ can be up to an isomorphism characterized as a countable pseudotree with the following two properties.
\begin{enumerate}
    \item For each $a \in \psi$ the set $D_\psi(a)$ is order isomorphic to the rational numbers $\mathbb Q$.
    \item 
    \begin{enumerate}
        \item
        For each $a \in \psi_\omega$ there exists a set 
        $\set{x_i \in \psi_\omega \mid a < x_i, i \in \omega}$ 
        such that $x_i \wedge x_j = a$ for each $i \neq j \in \omega$.
        \item $\psi_2$ is binary and 
        for each $a \in \psi_2$ there exist $x, y > a$  such that $x \wedge y = a$.
    \end{enumerate}
\end{enumerate}
\end{proposition}

\begin{definition}
    We say that $R \subset \psi$ is a \emph{ray} in $\psi$ if either
    \begin{itemize}
        \item $\inf R$ exists, denote it $o(R)$. 
        Then $o(R) \notin R$ and $R$ is a $\subseteq$-maximal linearly ordered subset of $\psi$ with infimum $o(R)$, or 
        \item $\inf R$ does not exist; in that case $R$ is a $\subseteq$-maximal linearly ordered subset of $\psi$.
    \end{itemize}
    We denote the set of all rays in $\psi$ by $\overline{\mathcal R}$.
    In either case, a ray is order-isomorphic to $\mathbb Q$.
\end{definition}

We will fix a 1-to-1 \emph{enumerating} function $f\colon \omega \to \psi \cup \overline{\mathcal R}$ with the following properties.

\begin{enumerate}
    \item $\psi \cap f[n]$ is a pseudotree for each $n \in \omega$,
    \item $\psi \subset f[\omega]$.
    \item If $f(n) = R \in \overline{\mathcal R}$ and $o(R)$ is defined, then $o(R) \in f[n] \cap \psi$.
    \item If $f(n) \in \psi$, 
        then $f(n) \in \bigcup (f[n] \cap \overline{\mathcal R})$.
    \item If $n \neq m$ and $f(n), f(m) \in \overline{\mathcal R}$, then $f(n) \cap f(m) = \emptyset$.
\end{enumerate}
Denote the enumerated rays as $\mathcal R =  f[\omega] \cap \overline{\mathcal R}$. 
Notice that necessarily $f(0) \in \mathcal R$ and $f(0)$ is the only element of $\mathcal R$ without an infimum.
Moreover $\psi = \bigcup \mathcal R$. 
The following lemma has a straightforward proof.

\begin{lemma}
    There exists an enumerating function $f$ with the stated properties.
\end{lemma}

We denote by $e \colon \psi \cup \mathcal R \to \omega$ the bijection inverse to $f$.
For $a \in \psi$ let $R(a)$ be the unique $R \in \mathcal R$ such that $a \in R$, 
and let $P(a) = \set{ R \in \mathcal R \mid R \cap D_\psi(a) \neq \emptyset}$ and 
$p(a) = \set{e(R) \mid R \in P(a)}$.
For $a \in \psi$ and $R \in P(a)$, $a \notin R$ define 
$a \wedge R$ to be the unique element $S$ of $P(a)$ such that $o(S) \in R$.


\begin{lemma}
    For every $a \in \psi$ is $P(a) \subseteq f[e(a)]$. 
    In particular $P(a)$ is finite.
\end{lemma}

\begin{proposition}\label{prop:Q-like}
    Suppose $\varphi' \in \binom{\psi}{\psi}$. 
    Then there exists $\varphi \in \binom{\varphi'}{\psi}$ 
    such that for every $R \in \mathcal R$ 
    if $R \cap \varphi \neq \emptyset$, then $R \cap \varphi$ is isomorphic to $\mathbb Q$.
\end{proposition}
    \begin{claim*}
        For every $b \in \varphi'$ there is $R \in \mathcal R$ such that $\set{x \in R \cap \varphi' \mid b < x }$ 
        contains a copy of $\mathbb Q$.
    \end{claim*}
        Fix $z \in \varphi'$, $b<z$. Since $\set{x \in \varphi' \mid b < x < z}$ contains a copy of $\mathbb Q$ and $P(z)$ is finite, 
        there must be $R \in P(z)$ as required.
        \claimdone
    For $n$ such that $f(n) \in \mathcal R$
    we inductively define order embeddings  $g(n) \colon f(n) \to \varphi'$; 
    we aim for $g = \bigcup \set{ g(n) \mid n \in \omega, f(n) \in \mathcal R}$ 
    being an embedding of $\psi$ into $\varphi'$ such that  $\varphi = g[\psi]$ is as required.
    Start by choosing arbitrary $R \in \mathcal R$ such that $\varphi' \cap R$ 
    contains a copy of $\mathbb Q$, and let $g(0) \colon f(0) \to \varphi' \cap R$ be an embedding.
    Suppose $f(n) \in \mathcal R$, and $g(i)$ is defined for all $i < n$, $f(i) \in \mathcal R$. 
    Write $h(n) = \bigcup \set{ g(i) \mid i <n}$ and $a' = o(f(n))$.
    Letting $a = h(n)(a')$, since $\varphi'$ is isomorphic to $\psi$ there exists $b \in \varphi'$ such that $o(b \wedge R(a)) = a$ 
    and $b \wedge x \leq a$ for all $x$ in the range of $h(n)$.
    The claim implies that there is $R \in \mathcal R$ such that we can fix an embedding $g(n) \colon f(n) \to \set{x \in R \cap \varphi' \mid b < x }$. 
    At the end $\bigcup \set{ g(n) \mid n \in \omega, f(n) \in \mathcal R}$ is as required.
\qed

\section{Indivisibility}

\noindent
A structure is \emph{indivisible} if the big Ramsey degree of singletons is equal to~$1$.
We prove that  $\psi$  is indivisible.
Let $\mathrm S$ be the single element pseudotree.

\begin{theorem}
     $\mathrm d(\psi : \mathrm S) =1 $.
\end{theorem}
\begin{proof}
    Suppose $C \subseteq  \psi$. We will find $\varphi \in \binom{\psi}{\psi}$
    such that either $\varphi \subseteq C$, 
    or $\varphi \cap C= \emptyset$.
    \begin{claim*}
        Either for every $b \in \psi$ there is $R \in \mathcal R$ such that $\set{x \in R \cap C \mid b < x }$ contains a copy of $\mathbb Q$, 
        or there is $y \in \psi$ such that 
        for every $b > y$ there is $R \in \mathcal R$ such that $\set{x \in R \setminus C \mid b < x }$ contains a copy of $\mathbb Q$.
    \end{claim*}
    Suppose that the first alternative fails and this is witnessed by $y \in \psi$, 
    we will argue that the second alternative holds.
    Choose any $b>y$, let $R \in \mathcal R$ be such that $o(R) = b$. 
    For any $u,v \in R$ with $u<v$, the set $\set{x \mid u < x < v}$ is isomorphic to $\mathbb Q$ 
    and therefore there is $w \in R \setminus C$, $u < w < v$, i.e.\ $R \setminus C$ 
    is a countable dense linear order and contains a copy of $\mathbb Q$.
    \claimdone
    The rest of the proof follows verbatim the proof of Proposition~\ref{prop:Q-like}; 
    if the first alternative of the claim occurs, construct $\varphi \subseteq C$. 
    In the second alternative start with choosing $g(0)$ such that 
    $x > y$ for each $x$ in the range of $g(0)$, and construct $\varphi$ disjoint with $C$.
\end{proof}

\section{A coloring of chains}

\noindent
Suppose $a, b \in \psi$ form a chain, that is\ $a < b$. 
We will denote the 2-element pseudotree isomorphic to $\set{a,b}$ as $\mathrm C$.
We define the value of a coloring function $c( a, b)$ as follows.

Case~1; $R(a) = R(b)$. Let $c(a,b ) = 0$ if $e(a) < e(b)$, and let $c(a,b ) = 1$ if $e(b) < e(a)$.

Case~2; $R(a) \neq R(b)$. 
In this case, it must be $e(o(b \wedge R(a))) < e(b \wedge R(a)) \leq e(R(b)) < e(b)$.
If $b \wedge R(a) = R(b)$ we let $c(a,b) = \mathrm u$. 
Otherwise, define $c(a,b)$ by distinguishing subcases.
\begin{enumerate}
    \item[(2)] $e(a) < e(o(b \wedge R(a)))$, let $c(a,b) = 2$
    \item[(3)] $a = o(b \wedge R(a))$, let $c(a,b) = 3$
    \item[($\mathrm e$)] $e(o(b \wedge R(a))) < e(a)< e(b \wedge R(a))$, let $c(a,b) = \mathrm e$
    \item[(4)] $e(b \wedge R(a)) < e(a) < e(R(b))$, let $c(a,b) = 4$ 
    \item[(5)] $e(R(b)) < e(a) < e(b)$, let $c(a,b) = 5$
    \item[(6)] $e(b) < e(a)$, let $c(a,b) = 6$
\end{enumerate}

\begin{theorem}
    Let $s = 7$ if $\psi$ is $\psi_2$, or $s = 7 \cup \set{\mathrm{e} }$ if $\psi$ is $\psi_\omega$.
    For every $\varphi \in \binom{\psi}{\psi}$ and 
    $c \in s$ there are $a < b \in \varphi$ such that $c(a,b) = c$.
\end{theorem}
\begin{proof}
    We can assume that $\varphi$ is as in the conclusion of Proposition~\ref{prop:Q-like}.
    Fix $a_2 \in \varphi$ arbitrary. 
    Since $R = R(a_2) \cap \varphi$ is isomorphic $\mathbb Q$, we can find $a_3, a_{\mathrm e} \in \varphi \cap R$ such that
    $a_2< a_3$, $a_{\mathrm e} < a_3$, and $e(a_2) < e(a_3) < e(a_{\mathrm e})$.
    
    There is $x \in \varphi \setminus R$, $a_3 < x$, $o(x \wedge R) = a_3$, and 
    if $\psi$ is $\psi_\omega$, then also $e(x \wedge R) > e(a_{\mathrm e})$.
    Find $a_4 \in D_\varphi(a_3) \cap R $ such that
    $e(x \wedge R) < e(a_4)$.
    Fix $y \in \varphi$, $y > x$ such that $e(R(y)) > e(a_4)$. 
    Find $a_{5} \in D_\varphi(a_3) \cap R$ such that
    $e(R(y)) < e(a_5)$.
    Finally fix $b \in \varphi \cap R(y)$ such that $e(b) > e(a_5)$,
    and $a_{6} \in D_\varphi(a_3) \cap R$ such that
    $e(b) < e(a_6)$.

    The construction yields $c(a_2, a_3) = 0$,
    $c(a_3, a_4) = 1$,
    $c(a_2, b) = 2$,
    $c(a_3, b) = 3$,
    $c(a_4, b) = 4$,
    $c(a_5, b) = 5$,
    $c(a_6, b) = 6$, 
    and $c(a_{\mathrm e}, b) = \mathrm e$ in case $\psi$ is $\psi_\omega$.
\end{proof}

\begin{corollary}
    $\mathrm d(\psi_2 : \mathrm C) \geq 7$, 
    $\mathrm d(\psi_\omega : \mathrm C) \geq 8$.
\end{corollary}

In fact, $\mathrm d(\psi_2 : \mathrm C) = 7$, 
$\mathrm d(\psi_\omega : \mathrm C) = 8$; this will be proved in forthcoming publications, e.g.~\cite{coding_paper}.
However, the methods to prove this are beyond the scope of this abstract.

\section{The oscillation on antichains}

\noindent
Denote the three element pseudotree consisting of a root and two incomparable nodes as~$\mathrm A$. 
In this section, we prove that $\mathrm d(\psi_2 : \mathrm A) = \mathrm d(\psi_\omega : \mathrm A) = \infty$.
In order to do this, we use the method of oscillation which has seen widespread use in set theory, see e.g.~\cite{stevo_osc}.
Our application was inspired by an argument about the product 
of Mathias posets; see~\cite[Observation~7]{Jorg_Luz}.

Suppose $u_0, u_1 \subset \omega$ are two disjoint finite sets. 
Let ${\simeq}$ be the equivalence relation on $u_0 \cup u_1$ defined by 
declaring $n \simeq m$ 
if there is $i\in 2$ such that $m,n \in u_i$ and there is no $k \in u_{1-i}$ such that $n<k<m$ or $m<k<n$. 
Define $\mathrm{osc}(u_0, u_1)$ to be the number of equivalence classes of $\simeq$
\[\mathrm{osc}(u_0, u_1) = \card{\set{{[n]}_{\simeq} \mid n \in u_0 \cup u_1}}.\]
For a pair of incomparable elements $a, b$ of $\psi$ define 
$\mathrm{os}(a,b) = \mathrm{osc}(p(a) \setminus p(b), p(b) \setminus p(a))$.
Notice that for any $a, b \in \psi$ the set $p(a) \cap p(b)$ is always an initial subset of $p(a)$.

\begin{theorem}\label{prop:osc}
    Let $\varphi \in \binom{\psi}{\psi}$.
    For every $\ell \in \omega \setminus 1$ there are incomparable $a,b \in \varphi$ such that $\mathrm{os}(a,b) = \ell$.
    In particular, $\mathrm d(\psi_2 : \mathrm A) = \mathrm d(\psi_\omega : \mathrm A) = \infty$.
\end{theorem}
\begin{proof}
     We again assume that $\varphi$ is as in the conclusion of Proposition~\ref{prop:Q-like}.
     We prove the theorem by induction on $\ell$.
     Choose any $a_1,b_1,r \in \varphi$ such that $a_1 \wedge b_1 = r < a_1, b_1$, 
     and $R(r) = R(a_1)$. 
     Now $p(a_1)$ is a proper initial subset of $p(b_1)$, i.e.\ $\mathrm{os}(a_1,b_1) = 1$.

     Suppose $a_\ell, b_\ell \in \varphi$ are incomparable and $\mathrm{os}(a_\ell,b_\ell) = \ell$. 
     We may assume without loss of generality $\max (p(a_\ell) \cup p(b_\ell)) \in p(a_\ell)$. 
     Since $R(b_\ell) \cap \varphi$ is isomorphic to $\mathbb Q$, there is $y \in \varphi \cap R(b_\ell)$ 
     such that $e(y) > e(a_\ell)$. Let $b_{\ell+1}$ be an element of $\varphi$ such that $o(b_{\ell+1} \wedge R(b_\ell)) = y$. 
     Then $p(b_{\ell+1})$ is a proper end-extension of $p(b_{\ell})$, $(p(b_{\ell+1}) \setminus p(b_{\ell})) \cap e(a_\ell) = \emptyset$, 
     and for $a_{\ell+1} = a_{\ell}$ we get $\mathrm{os}(a_{\ell+1},b_{\ell+1}) = \ell+1$. 
\end{proof}


\section{The generic $C$-relation}

\noindent
The $C$-relation introduced in~\cite{Adeleke-Neuman} is a ternary relation
with axioms that describe the behavior of leaves of a finite pseudotree. 
More precisely, a finite set $L$ equipped with a ternary relation $C$ is a $C$-relation structure 
(or just a \emph{$C$-relation}) 
if there is a finite pseudotree $T$ such that
the set $L$ consists of maximal elements of $T$, and $C(a;b,c)$ exactly when $a \wedge b < b \wedge c$.

The class of finite $C$-relations and binary $C$-relations (i.e.\ corresponding to binary pseudotrees) 
are both Fraïssé classes, denote their limits 
$\mathscr C_\omega$ and $\mathscr C_2$, both of these objects can be characterized 
as countable $C$-relations satisfying certain first order axioms, see~\cite{Bodirsky+}. 
Sam Braunfeld observed~\cite{Sam-oral} that our results for $\psi$ imply that 
$\mathscr C_\omega$ and $\mathscr C_2$ do not have finite big Ramsey degrees. 

We will sketch the argument.
Let $\mathrm{C_4}$ be the $C$-relation on 4 elements $\set{\mathrm{a,b,c,d}}$ uniquely determined by declaring 
$C(\mathrm{a; c,d})$, $C(\mathrm{b; c,d})$, and $C(\mathrm{c; a,b})$.

\begin{theorem}[Braunfeld]
    $\mathrm d(\mathscr C_2 : \mathrm{C_4}) = \mathrm d(\mathscr C_\omega : \mathrm{C_4}) = \infty$.
\end{theorem}
\begin{proof}
    For $R \in \mathcal R$ choose any $r\in R$ and let $D(R) = R \cup D(r)$, and 
    for $R \neq S \in \mathcal R$ let $R \wedge S = \max (D(R) \cap D(S))$. 
    For $R,S,T \in \mathcal R$ define $C(R; S,T)$ if $R \wedge S < S \wedge T$ or $R \neq S = T$.
    Using the axioms from~\cite{Bodirsky+} it is easy to verify that $(\mathcal R, C)$ is isomorphic 
    to $\mathscr C_2$ if $\psi = \psi_2$, and $\mathscr C_\omega$ if $\psi = \psi_\omega$.
    Moreover, whenever $\mathcal S \in \binom{\mathcal R}{\mathcal R}$
    then $\set{S \wedge R \mid S,R \in \mathcal S} \in \binom{\psi}{\psi}$. 

    Let $X = \set{a,b,c,d} \in \binom{\mathcal R}{\mathrm{C_4}}$. 
    Define $c(X) = \mathrm{os}(a\wedge b, c\wedge d)$.
    It follows from Theorem~\ref{prop:osc} 
    that whenever $\mathcal S \in \binom{\mathcal R}{\mathcal R}$ 
    and $\ell \in \omega \setminus 1$, 
    there is $X \in \binom{\mathcal S}{\mathrm{C_4}}$ 
    such that $c(X) = \ell$.
\end{proof}

\section*{Acknowledgments}

\noindent
The first author was supported by project 25-15571S of the Czech Science Foundation (GAČR) and by the Czech Academy of Sciences CAS (RVO 67985840). 

The second author was supported by the Austrian Science Fund (FWF) project P34603.

The third author
was supported by the Fonds zur F\"orderung der wissenschaftlichen Forschung, Lise Meitner grant M 3037-N, as well as the Research Project of National relevance ``PRIN2022\_DIMONTE - Models, sets and classifications(realizzato con il contributo del progetto PRIN 2022 - D.D. n. 104 del 02/02/2022 – PRIN2022\_DIMONTE - Models, sets and classifications - Codice 2022TECZJA\_003 - CUP N. G53D23001890006. ``Finanziato dall'Unione Europea – Next-GenerationEU – M4 C2 I1.1'')

\bibliography{references}
\bibliographystyle{amsalpha}

\end{document}